\documentclass[letterpaper,12pt,reqno]{amsart}
\usepackage{amssymb}

\usepackage{stmaryrd}
\usepackage{amsmath,color,amscd,yhmath}
\usepackage{mathrsfs}
\usepackage{eucal}
\usepackage{amsfonts}
\usepackage[all]{xy}
\usepackage{pstricks}
\usepackage{pstricks,pst-node}
\usepackage{tikz}
\usepackage{bbm}

\DeclareFontFamily{OT1}{pzc}{}
\DeclareFontShape{OT1}{pzc}{m}{it}{<-> [1.15] rpzcmi}{}
\DeclareMathAlphabet{\mathzc}{OT1}{pzc}{m}{it}
\DeclareMathAlphabet{\mathpzc}{OT1}{pzc}{m}{it}

\DeclareSymbolFont{largesymbol}{OMX}{yhex}{m}{n}
\DeclareMathAccent{\Widehat}{\mathord}{largesymbol}{"62}

\newcommand{\Z}{\mathbb{Z}}

\def\la{{\lambda}}
\def\al{{\alpha}}

\newtheorem{theorem}{Theorem}[section]
\newtheorem{lemma}[theorem]{Lemma}

\newtheorem{corollary}[theorem]{Corollary}

\theoremstyle{definition}
\newtheorem{definition}[theorem]{Definition}
\newtheorem{example}[theorem]{Example}

\newtheorem{remark}[theorem]{Remark}

\numberwithin{equation}{theorem}

\def\KK{{\mathzc K\kern0pt}}
\def\vv{{\mathzc v\kern.5pt}}

\def\aq{/\kern-2pt/}

\def\diag{\operatorname{diag}}

\def\La{{\Lambda}}

\def\bse{{\boldsymbol e}}
\def\up{{\boldsymbol{\upsilon}}}

\def\la{{\lambda}}
\def\al{{\alpha}}

\def\fS{{\mathfrak S}}

\def\bfe{{\boldsymbol{e}}}
\def\bfj{{{\bf j}}}

\def\bfa{{{\bf a}}}
\def\bfc{{{\bf c}}}

\def\sSmn{{{\mathcal S^{m|n}}}}
\def\wsSmn{{\widetilde{\mathcal S}^{m|n}}}
\def\SgAj{{A(\bfj)}}

\def\upsh{{\boldsymbol{\upsilon}}_h}
\def\upsi{{\boldsymbol{\upsilon}}_{h+1}}
\def\ups{{\boldsymbol{\upsilon}}}
\def\tiX{X}
\def\bbmk{{\mathbbm k}}
\def\ulA{{\underline A}}

\begin{document}

\baselineskip16pt
\def\hei{\relax}

\title{
The regular representation  of $U_\up(\mathfrak{gl}_{m|n})$ 
}

\author{Jie Du and Zhongguo Zhou}

\address{J.D., School of Mathematics and Statistics,
University of New South Wales, Sydney NSW 2052, Australia}
\email{j.du@unsw.edu.au}
\address{Z.Z., College of Science, Hohai University, Nanjing, China}
\email{zhgzhou@hhu.edu.cn}

\date{\today}

\thanks{
The second author would like to thank UNSW for its hospitality during his one year visit and thank the Jiangsu Provincial Department of Education for financial support.}

\maketitle
\begin{abstract}Using quantum differential operators, we construct a super representation of $U_\up(\mathfrak{gl}_{m|n})$ on a certain polynomial superalgebra. We then extend the representation to its formal power series algebra which contains a $U_\up(\mathfrak{gl}_{m|n})$-submodule isomorphic to the regular representation  of $U_\up(\mathfrak{gl}_{m|n})$. In this way, we obtain a presentation of $U_\up(\mathfrak{gl}_{m|n})$ by a basis together with explicit multiplication formulas of the basis elements by generators.
\end{abstract}
\medskip

\section{Introduction}

Arising from the natural representation $V$ of the quantum supergroup $U_\up(\mathfrak{gl}_{m|n})$, the investigation on the tensor products $V^{\otimes r}$ for all $r\geq 0$ has recently produced interesting outcomes. For example, the root-of-unity theory resulted in a new proof for the quantum Mullineux conjecture (see \cite{DLZ}). On the other hand, the generic theory on $\up$-Schur superalgebras, which are homomorphic images of the representations $U_\up(\mathfrak{gl}_{m|n})\to\text{End}(V^{\otimes r})$, gives rise to a new construction for $U_\up(\mathfrak{gl}_{m|n})$ itself (see \cite{DG}).
This latter work extends the geometric realisation of quantum  $\mathfrak{gl}_{n}$, given by Beilinson--Lusztig--MacPherson (BLM) in \cite{BLM}, to the super case. The BLM work has also been generalised to the quantum affine case \cite{DDF,DF1} and the case for the other classical types
\cite{BKLW,FL}.

Furthernore, in the nonsuper case, there are other representations of $U_\up(\mathfrak{gl}_{n})$ arising from the symmetric and exterior algebras $S(V)$ and $\Lambda(V)$ of the natural representation $V$; see, e.g., \cite{H} and \cite[\S\S5A.6-7]{jc}, where the module actions are defined by using certain quantum differential operators. Can these representations be used to determine the structure of a quantum supergroup?
We will provide an affirmative answer in this paper.

We will start with the natural super representations $V=V_0\oplus V_1$ of $U_\up(\mathfrak{gl}_{m|n}).$ We first introduce two types of symmetric superalgebras $\mathcal S_{0|1}(V)=S(V_0)\otimes \Lambda(V_1)$ and $\mathcal S_{1|0}(V)=\Lambda(V_0)\otimes S(V_1)$ and their mixed tensor product $\sSmn(V)$. The supermodule structure on each of them is defined via quantum differential (super) operators. We then extend the supermodule structure to the formal power series algebra $\wsSmn(V)$ which is naturally a $U_\up(\mathfrak{gl}_{m|n})$-module.
We will extract a submodule from $\wsSmn(V)$ which naturally possesses a supermodule structure. We prove in the main theorem (Theorem \ref{genmod}) that this supermodule is isomorphic to the regular representation of $U_\up(\mathfrak{gl}_{m|n})$.
Thus, we obtain a new presentation for $U_\up(\mathfrak{gl}_{m|n})$ (cf. Lemma \ref{key lemma}).
 Surprising enough, this presentation from the regular representation of $U_\up(\mathfrak{gl}_{m|n})$ coincides with the one from \cite[Lemma 5.3]{BLM} when $n=0$ and with the one as given in \cite[Thm 8.4]{DG} in general, both of which were obtained either by a geometric method involving quantum Schur algebras or by an algebraic method involving quantum Schur superalgebras.

\section{The quantum supergroup $U_\up(\mathfrak {gl}_{m|n})$ and differential operators}

  For fixed non-negative integers $m,n$ with $m+n>0$, let $ [1,m+n]:=\{1,2,\cdots, m+n\},$
and define the parity function $\widehat{\ }:[1,m+n]\to\mathbb Z_2,i\mapsto \widehat i$ by
$$
\widehat i= \begin{cases}
             0, & \mbox{  if  } 1\leq i\leq m;\\
            1, & \mbox{  if  } m+1\leq i\leq m+n.
           \end{cases}
$$
We will always regard $\mathbb Z_2=\{0,1\}$ as a subset of $\mathbb N$ unless it is used for the grading of a super structure. For any superspace $V$ and a homogeneous element $v\in V$, we often use $\widehat v$ to denote its parity.

Let $\{ \bfe_1,\bfe_2,\cdots, \bfe_{m+n}\}$ be the  standard basisfor $\mathbb Z^{m+n}$, and
define the ``super dot product'' on $\mathbb Z^{m+n}$ by
\begin{equation}\label{superdp}
\bfe_i*\bfe_j=(\bfe_i,\bfe_j)_s=(-1)^{\widehat{i}}\delta_{ij}.
\end{equation}
Let $\mathbb Q(\up)$ be the field of rational functions in indeterminate $\up$ and let
$$\up_h=\up^{(-1)^{\hat h}}\quad(h\in[1, m+n]),
\quad[a]^!=[1][2]\cdots[a],\quad[i]=\frac{\up^i-\up^{-i}}{\up-\up^{-1}}\;\;(a\in\mathbb N).$$
Let $[i]_q$ denote the value at $q$.

Define the super (or graded) commutator on the homogeneous elements
$X,Y$ for an (associative) superalgebra by
$$
[X,Y]=[X,Y]_s=XY-(-1)^{\widehat X \widehat Y}YX.
$$
   The following quantum enveloping superalgebra $U_\up(\mathfrak {gl}_{m|n})$
is defined in \cite{zhang}.
\begin{definition}\label{Uglmn}
(1) The quantum enveloping superalgebra $U_\up(\mathfrak{gl}_{m|n})$
over $ \mathbb Q(\up)$  is generated by
\[\begin{cases}
\text{even generators: }&E_{h}, F_{h},  K_{j}^{\pm  1},\;1\leq h,j\leq m+n,h\neq m, m+n;\\
\text{odd generators: }&E_m,F_m.\end{cases}
\]
These elements are subject to the following relations:
\begin{enumerate}
\item[(QG1)]
$ K_aK_b=K_bK_a,\ K_aK_a^{-1}=K_a^{-1}K_a=1; $
\item[(QG2)]
$ K_aE_{b}=\up^{(\varepsilon_{a}, \alpha_{b})_s}E_{b}K_a,\
K_aF_{b}=\up^{(\varepsilon_{a}, -\alpha_{b})_s}F_{b}K_a; $
\item[(QG3)]
$ [E_{a},
F_{b}]=\delta_{a,b}\frac{K_aK_{a+1}^{-1}-K_a^{-1}K_{a+1}}{\up_a-\up_a^{-1}};
$
\item[(QG4)]
$ E_{a}E_{b}=E_{b}E_{a}, \  F_{a}F_{b}=F_{b}F_{a},$ if $|a-b|>1 ;$
\item[(QG5)]
For $ a\neq m$ and $|a-b|=1,$
\begin{equation}
\begin{aligned}\notag
&E_{a}^2E_{b}-(\up_a+\up_a^{-1})E_{a}E_{b}E_{a}+E_{b}E_{a}^2=0,\\
&F_{a}^2F_{b}-(\up_a+\up_a^{-1})F_{a}F_{b}F_{a}+F_{b}F_{a}^2=0;
\end{aligned}
\end{equation}
\item[(QG6)]
$E_m^2=F_m^2=[E_{m}, E_{m-1,m+2}]=[F_{m}, E_{m+2,m-1}]=0,$ where
\begin{equation}
\begin{aligned}\notag
E_{m-1,m+2}&=E_{m-1}E_{m}E_{m+1}-\up E_{m-1}E_{m+1}E_{m} -\up^{-1}
E_{m}E_{m+1}E_{m-1}+
E_{m+1}E_{m}E_{m-1},\\
E_{m+2,m-1}&=F_{m+1}F_{m}F_{m-1}-\up^{-1} F_{m}F_{m+1}F_{m-1} -\up
F_{m-1}F_{m+1}F_{m}+
F_{m-1}F_{m}F_{m+1}.\\
\end{aligned}
\end{equation}
\end{enumerate}

(2) Let $\bbmk$ be a field and let $q\in \bbmk$, $q\neq0$ and $q^2\neq1$. Then, with $\up$ replaced by $q$, we may similarly define the quantum group $U_q(\mathfrak {gl}_{m|n})$ over $\bbmk$ (see \cite{jc, zhang}).
\end{definition}
Note that, if $n=0$, then (QG1)--(QG5) form a presentation for the quantum group $U_\up(\mathfrak {gl}_{m})$.

A Hopf algebra structure  on $U_\up(\mathfrak {gl}_{m|n})$ is
defined  (see \cite[Section II]{zhang}) by:
\begin{equation}\label{coalg}
\begin{aligned}
& \Delta(K_i)=K_i\otimes K_i,\\
&\Delta(E_i)=E_i\otimes {\widetilde K}_{i} +1\otimes E_i, \quad
   \Delta(F_i)=F_i\otimes 1+{\widetilde K}_i^{-1}\otimes F_i ,\\
& \varepsilon(K_i)=1,  \quad  \varepsilon(E_i)= \varepsilon(F_i)=0,\\
& S(K_i)=K_i^{-1}, \quad S(E_i)=-E_i {\widetilde K}_{i}^{-1}, \quad
S(F_i)=-{\widetilde K}_{i} F_i,
\end{aligned}
\end{equation}
where ${\widetilde K}_{i}=K_iK_{i+1}^{-1}.$

Representations of $U_\up(\mathfrak {gl}_{m|n})$ have been investigated in \cite{zhang} (see also \cite{DLZ} for representations of its hyperalgebra at roots of unity).
We will need two special $U_\up(\mathfrak {gl}_{m|n})$-supermodules in the next section for our construction. They are built on the following two $U_\up(\mathfrak {gl}_N)$-modules defined by
quantum differential operators.

\begin{example}\label{QDO}
Let $V$ be a vector space over a field $\bbmk$ of dimension $N$ and let  $\bbmk[x_1,x_2,\cdots,x_N]$ be the polynomial algebra over $\bbmk$ in indeterminates $x_1,\ldots,x_N$.

(1) Let $S(V)$ be the {\it symmetric algebra} on $V$, identified as  $S(V)=\bbmk[x_1,x_2,\cdots,x_N]$. Following \cite[5A.6]{jc}, we define {\it quantum differential operators} $\mathcal D_i:S(V)\to S(V)$ by
$$\mathcal D_i(x_1^{a_1}x_2^{a_2}\cdots x_N^{a_N})=\begin{cases}[a_i]_qx_1^{a_1}x_2^{a_2}\cdots x_i^{a_i-1}\cdots  x_N^{a_N},&\text{if }a_i\geq1;\\
0,&\text{otherwise}.
\end{cases}$$
We also introduce algebra automorphism $\mathcal K_i:S(V)\to S(V)$ by setting
$$\mathcal K_i(x_1^{a_1}x_2^{a_2}\cdots x_N^{a_N})=q^{a_i}x_1^{a_1}x_2^{a_2}\cdots x_N^{a_N}.$$
Let $\mathcal E_i=x_i\circ\mathcal D_{i+1}$ and $\mathcal Fi=x_{i+1}\circ\mathcal D_{i}$. Then, by \cite[Prop. 5A.6]{jc},
 the following map
$$E_i\longmapsto\mathcal E_i,\;\; F_i\longmapsto\mathcal F_i,\;\; K_j\longmapsto\mathcal K_j$$
for all $1\leq i,j\leq N$ ($i\neq N$) defines an algebra  homomorphism from $U_q(\mathfrak {gl}_N)$ to the endomorphism algebra of $S(V)$. Hence, $S(V)$ becomes a $U_q(\mathfrak {gl}_N)$-module (cf. \cite[Thm 4.1(A)]{H}).

(2) Let $\La(V)$ be the {\it exterior algebra} on $V$. In this case, we may identify $\La(V)$ with the Grassman superalgebra
$\La(d_1,\ldots,d_N)$ with {\it odd} generators $d_1,\ldots,d_N$ and relations
$$d_i^2=0\; (1\leq i\leq N),\quad d_id_j=-d_jd_i\;(1\leq i\neq j\leq N).$$
Thus, $\La(V)$ has a basis $d^\bfa:=d_1^{a_1}\cdots d_N^{a_N}$, $\bfa\in\mathbb Z_2^N$.
Define a $U_q(\mathfrak {gl}_N)$-action on $\La(V)$ by
$$K_i.d^\bfa=q^{a_i}d^\bfa,\;\;E_h .d^\bfa=\begin{cases} d^{\bfa+\al_h},&\text{if }a_{h+1}>0;\\
0,&\text{otherwise,}\end{cases}\;\;F_h. d^\bfa=\begin{cases} d^{\bfa-\al_h},&\text{if }a_{h}>0;\\
0,&\text{otherwise,}\end{cases}$$
for all $1\leq h,i\leq N$, $h\neq N$. It is direct to check that all relations (QG1-5) are satisfied. Hence, $\La(V)$ becomes a $U_q(\mathfrak {gl}_N)$-module (cf. \cite[\S\S2,4]{H}).
\end{example}

\section{The polynomial superalgebra $\sSmn(V)$ as a $U_\up(\mathfrak{gl}_{m|n})$-supermodule}
We generalize the constructions of the module structures on
symmetric and exterior algebras to the supergroup
$U_\up(\mathfrak{gl}_{m|n}).$

Consider the natural representation on the superspace $V=V_0\oplus V_1$ of $\mathfrak{gl}_{m|n}(\bbmk)$ where $\dim V_0=m$ and $\dim V_1=n$. We will consider two superalgebras in the notation of Example \ref{QDO}:
$$\aligned
S(V_0)\otimes \La(V_1)&= \bbmk[x_1,\ldots,x_m,d_1,\ldots,d_n],\\
\La(V_0)\otimes S(V_1)&= \bbmk[d_1,\ldots,d_m,x_1,\ldots,x_n].\endaligned$$
These are known as {\it polynomial superalgebras} with even generators $x_i$ and odd generators $d_j$. By Example \ref{QDO}, both algebras are also $U_q(\mathfrak{gl}_{m})\otimes U_{q}(\mathfrak{gl}_{n})$-modules

We now assume $\bbmk=\mathbb Q(\up)$. In order to introduce supermodule structure for $U_\up(\mathfrak{gl}_{m|n})$,
we set
\begin{equation}
\begin{aligned}
&\mathcal S_{0|1}=\mathcal S_{0|1}(V):=\mathbb Q(\up)[X_1,X_2,\cdots, X_{m+n}]\text{ with }X_i=x_i, X_{m+j}=d_j,\\
&\mathcal S_{1|0}=\mathcal S_{1|0}(V):=\mathbb Q(\up)[X_1,X_2,\cdots, X_{m+n}]\text{ with }X_i=d_i, X_{m+j}=x_j,
\end{aligned}
\end{equation}
where $1\leq i\leq m,1\leq j\leq n$.
We use divided powers to denote their monomial bases:
$$
 X^{(\bfa)}=X_{1}^{(a_{1})}X_{2}^{(a_{2})}\cdots X_{m+n}^{(a_{m+n})},
$$
where $\bfa=(a_1,\cdots, a_{m+n})\in \mathbb N^{m}\times\mathbb Z_2^n$ for $\mathcal S_{0|1}$, $\bfa=(a_1,\cdots, a_{m+n})\in\mathbb Z_2^m\times \mathbb N^{n}$ for $\mathcal S_{1|0}$,
and $ X_i^{(a_i)}=\frac{X_i^{a_i}}{[a_i]!}$.

For the superspace structure, we have, for $i\in\mathbb Z_2$,
 ${X^{(\bfa)}}\in(\mathcal S_{0|1})_i$ if and only if $\widehat\bfa:=\sum_{j=1}^na_{m+j}\equiv i(\text{mod}\; 2)$, while ${X^{(\bfa)}}\in(\mathcal S_{1|0})_i$ if and only if $\widehat\bfa:=\sum_{j=1}^ma_{j}\equiv i(\text{mod}\; 2)$.

%

 As algebras, both $\mathcal S_{0|1}$ and $\mathcal S_{1|0}$ have a graded structure $\mathcal S_{0|1}(r)$ and $\mathcal S_{1|0}(r)$ for all $r\geq0$, where $\mathcal S_{0|1}(r)$ (resp., $\mathcal S_{1|0}(r)$) is the $r$-th homogeneous component spanned by all $X^{(\bfa)}$ with $\deg(X^{(\bfa)})=r$.  Here $\deg(X^{(\bfa)}):=\sum_{i=1}^{m+n}a_i$.

We now define the following actions on $\mathcal S_{0|1}$ and $\mathcal S_{1|0}$ by the same rules:
\begin{equation}\label{eaction}
\aligned
K_i. X^{(\bfa)}&=\ups_i^{a_i}X^{(\bfa)},\\
E_h. X^{(\bfa)}&=\left\{\begin{aligned}
 &[a_h+1]X^{(\bfa+\alpha_{h})}, & \mbox{if } a_{h+1}>0;\\
 & 0,                          &\mbox{otherwise}.
\end{aligned}
\right.\\
F_h. X^{(\bfa)}&=\left\{\begin{aligned}
& [a_{h+1}+1]X^{(\bfa-\alpha_{h})},&\mbox{if } a_{h}>0;\\
&0, &\mbox{otherwise}.
\end{aligned}
\right.,\\
\endaligned
\end{equation}
where $1\leq h,i\leq m+n$, $h\neq m+n$, and $\alpha_h=\bse_h-\bse_{h+1}$ are ``simple roots''.
Note that, for the even generators, the action above (on quantum divided powers) coincides with those given in Example \ref{QDO}.
Thus, both $\mathcal S_{0|1}$ and $\mathcal S_{1|0}$ are $U_\up(\mathfrak{gl}_{m})\otimes U_{\up^{-1}}(\mathfrak{gl}_{n})$-modules under the action above.
%
%

\begin{lemma}\label{polmod}
Both $\mathcal S_{0|1} $ and $\mathcal S_{1|0}$ are $U_\up(\mathfrak{gl}_{m|n})$-supermodules under the actions
above. In particular, their homogeneous components  $\mathcal S_{0|1}(r), \mathcal S_{1|0}(r)
,r\geq 0$ are all subsuperbmodules.
\end{lemma}
\begin{proof}We only need to verify the
defining relations that involve the odd generators.

We only prove the case  for $\mathcal S_{0|1}.$ It is easy to verify the relations
(QG2) and (QG4). Note that the actions of $E_m, F_m$  is consistent with
 those for even generators, so (QG5) holds.
It remains to check (QG3) and (QG6).

The relations $[E_m,E_b]=0=[E_b,F_m]$ with $m\neq b$ in (QG3) are clear. Assume now $a=b=m$.
Let $\bfa=(a_1,\cdots, a_{m+n})\in \mathbb N^{m}\times\mathbb Z_2^{n}.$
If $a_{m+1}=1$ then
\begin{equation}\notag
\begin{aligned}
& (E_mF_m+F_mE_m).X^{(\bfa)}=F_mE_m.X^{(\bfa)}=[a_m+1]F_mX^{(\bfa+\alpha_{m})}=[a_m+1]X^{(\bfa)}.\\
&\frac{K_mK_{m+1}^{-1}-K_m^{-1}K_{m+1}}{\up-\up^{-1}}.X^{(\bfa)}
=\frac{\up^{a_m+1}-\up^{-a_m-1}}{\up-\up^{-1}}.X^{(\bfa)}=[a_m+1]X^{(\bfa)}.\\
\end{aligned}
\end{equation}
If $a_{m+1}=0, a_m>0$ then
\begin{equation}\notag
\begin{aligned}
 (E_mF_m+F_mE_m).X^{(\bfa)}&=E_mF_m.X^{(\bfa)}=E_mX^{(\bfa-\alpha_{m})}=[a_m]X^{(\bfa)}.\\
&=\frac{K_mK_{m+1}^{-1}-K_m^{-1}K_{m+1}}{\up-\up^{-1}}.X^{(\bfa)}.\\
\end{aligned}
\end{equation}
If $a_{m+1}=0, a_m=0$ then
\begin{equation}\notag
\begin{aligned}
& (E_mF_m+F_mE_m).X^{(\bfa)}=0=\frac{K_mK_{m+1}^{-1}-K_m^{-1}K_{m+1}}{\up-\up^{-1}}.X^{(\bfa)} .
\end{aligned}
\end{equation}
So, in all three cases, we obtain
$$(E_mF_m+F_mE_m).X^{(\bfa)}=\frac{K_mK_{m+1}^{-1}-K_m^{-1}K_{m+1}}{\up-\up^{-1}}.X^{(\bfa)},$$
for all $\bfa\in \mathbb N^{m}\times\mathbb Z_2^{n}$, proving (QG3).

Finally,  we prove the four relations in (QG6).
As $a_{m+1}\leq 1$, we have
$E_m^2.X^{(\bfa)}=0=F_m^2.X^{(\bfa)}$ for all $\bfa$. For the other two relations,
if $a_{m+1}=1, a_{m+2}=1$ then
\begin{equation}\notag
\begin{aligned}
 E_mE_{m-1,m+2}.X^{(\bfa)}
&=(E_mE_{m-1}E_{m}E_{m+1}-\up E_mE_{m-1}E_{m+1}E_{m}\\
&\qquad-\up^{-1}E_mE_{m}E_{m+1}E_{m-1}+E_mE_{m+1}E_{m}E_{m-1})X^{(\bfa)}\\
&=(-\up E_mE_{m-1}E_{m+1}E_{m}
+E_mE_{m+1}E_{m}E_{m-1})X^{(\bfa)}\\
&=(-\up
[a_{m}+1][a_{m-1}+1][a_{m}+1]+[a_{m-1}+1][a_{m}][a_{m}+1])X^{(\bfa')},
\end{aligned}
\end{equation}
where $\bfa'=\bfa+\alpha_{m-1}+2\alpha_m+\alpha_{m+1}=\bfa+\bfe_{m-1}+\bfe_{m}-\bfe_{m+1}-\bfe_{m+2}$. On the other hand,
\begin{equation}\notag
\begin{aligned}
E_{m-1,m+2}E_m.X^{(\bfa)}
&=(E_{m-1}E_{m}E_{m+1}E_m-\up E_{m-1}E_{m+1}E_{m}E_m\\
&\qquad-\up^{-1}E_{m}E_{m+1}E_{m-1}E_m+E_{m+1}E_{m}E_{m-1}E_m)X^{(\bfa)}\\
&=(E_{m-1}E_{m}E_{m+1}E_m-\up^{-1}E_{m}E_{m+1}E_{m-1}E_m)X^{(\bfa)}\\
&=([a_{m}+1][a_{m}+2][a_{m-1}+1]-\up^{-1}[a_{m}+1][a_{m-1}+1][a_{m}+1])X^{(\bfa')}
\end{aligned}
\end{equation}
Since $[a_{m}]+[a_{m}+2]-(\up+\up^{-1})[a_{m}+1]=0$, it follows that
\begin{equation}\notag
\begin{aligned}
&( E_{m-1,m+2}E_m+E_mE_{m-1,m+2}).X^{(\bfa)}\\
=&[a_{m-1}+1][a_{m}+1]([a_{m}]+[a_{m}+2]-(\up+\up^{-1})[a_{m}+1])X^{(\bfa')}=0.
\end{aligned}
\end{equation}
 If $a_{m+1}=0$ or $a_{m+2}=0$ then
$E_{m-1,m+2}E_m.X^{(\bfa)}=0=E_mE_{m-1,m+2}.X^{(\bfa)}$ by the
definition of the actions. The last case can be proved similarly. This proves (QG6).
\end{proof}
The following result is a super analog of a result stated at the end of \cite[5A.7]{jc} (see also \cite[Thms 4.1(A), 4.2]{H}). Recall from, say, \cite{DLZ} that irreducible weight $U_\up(\mathfrak{gl}_{m|n})$-modules are indexed by
$$\mathbb N^{m|n}_{++}=\{\la\in\mathbb N^{m+n}\mid \la_1\geq\cdots\geq\la_m,\la_{m+1}\geq\cdots\geq\la_{m+n}\}.$$
\begin{corollary}
Let $\Delta(r\bfe_1)$ (resp., $\nabla(r\bfe_{m+n})$) be the irreducible weight $U_\up(\mathfrak{gl}_{m|n})$-module of highest (resp., lowest) weight $r\bfe_1$ (resp., $r\bfe_{m+n}$). Then, there are $U_\up(\mathfrak{gl}_{m|n})$-module isomorphisms:
$$\mathcal S_{0|1}(r)\cong \Delta(r\bfe_1),\quad \mathcal S_{1|0}(r)\cong \nabla(r\bfe_{m+n}).$$
\end{corollary}
\begin{proof}
Let $\lambda=r\bse_1\in \mathbb N^{m|n}_{++}$. Then
$X^{(\lambda)}\in \mathcal S_{0|1}(r)$ is a highest weight vector, since, for any
$\bfa=(a_1,\cdots, a_{m+n})\in \mathbb N^{m}\times\mathbb Z_2^n$ with $|\bfa|=r$,
$r\bse_1-\bfa=a_2(\bse_1-\bse_2)+a_3(\bse_1-\bse_3)+\cdots+a_{m+n}(\bse_1-\bse_{m+n})$ and
$$
(F_1^{(a_2)}\cdot F_2^{(a_3)}F_1^{(a_3)}\cdot\cdots\cdot F_{m+n-1}^{(a_{m+n})}\cdots F_{1}^{(a_{m+n})}).X^{(\lambda)}=X^{(\bfa)}.
$$
Hence, $\mathcal S_{0|1}(r)$ is generated by an highest weight vector. On the other hand, a reversed sequence in the $E_i^{(a_i)}$'s send $X^{(\bfa)}$ back to $X^{(\la)}$. Thus, $\mathcal S_{0|1}(r)$ is irreducible.
The proof for $\mathcal S_{1|0}(r)$ is similar.
\end{proof}
Consider the tensor product
\begin{equation}\label{sSmn}
\sSmn=\sSmn(V)=(\mathcal S_{0|1})^{\otimes m}\otimes(\mathcal S_{1|0})^{\otimes n}\cong\mathbb Q(\up)[X_{i,j}]_{1\leq i,j\leq m+n},
\end{equation}
where $X_{i,j}$ denotes the $i$th generator of the $j$-th tensor factor. Thus,
we may regard $\sSmn$ as the polynomial superalgebra as indicated by the right hand side of \eqref{sSmn}, which has even generators $X_{i,j}$ for all $i,j$ with {  $ \hat{i}+\hat{j}=0$} and odd generators $X_{i,j}$ for all $i,j$ with
  $ \hat{i}+\hat{j}=1$.
In particular, we may describe the monomial basis for $\sSmn$ in terms of the following matrix set:
\begin{equation}\label{blocks}
M(m|n)=\Bigg\{{X\;\;Q\choose Q'\;\; Y}\Bigg|\;{X\in M_m(\mathbb N),Q\in M_{m\times n}(\mathbb Z_2)\atop Q'\in M_{n\times m}(\mathbb Z_2), Y\in M_n(\mathbb N)}\Bigg\}.
\end{equation}
 For $A=(a_{i,j})\in M(m|n),$ let
 $${\bfc_{i}}=\bfc_i(A)=(a_{1,i},a_{2,i},\ldots,a_{m+n,i})$$ be the $i$-th column of $A$ and let
\begin{equation}\notag
X^{[A]}
 :=X^{(\bfc_{1})}\otimes X^{(\bfc_{2})} \cdots \otimes
 X^{(\bfc_{m+n})}.
\end{equation}
The parity of $X^{[A]}$ is given by $\widehat A:=\sum_{\widehat i+\widehat j=1}a_{i,j}$.

Via the coalgebra structure \eqref{coalg} of $U_\up(\mathfrak{gl}_{m|n})$,  $\sSmn$ becomes a $U_\up(\mathfrak{gl}_{m|n})$-module (see the lemma below). 
Recall also the sign rule: for supermodules $V_1, V_2$ over a superlagbera $U$, if $u_1, u_2\in U,
v_i\in V_i$ with $u_2,v_1$ homogeneous, then
\begin{equation*}
\begin{aligned}\label{signte}
(u_1\otimes u_2). (v_1\otimes
v_2)=(-1)^{\widehat{u_2}\,\widehat{v_1}}u_1v_1\otimes u_2v_2.
\end{aligned}
\end{equation*}

For $A\in M(m|n),$ $i\in[1,m+n]$, let
\begin{equation}\label{sigmak}
\sigma(i,A)=
\begin{cases}
\displaystyle\sum_{s>m,t<i}a_{s,t},&\text{ if }1\leq i\leq m;\\
\displaystyle\sum_{s>m,t\leq
m}a_{s,t}+\sum_{s\leq    m,\, m< t<i}a_{s,t},&\text{ if }m+1\leq i\leq m+n,\\
\end{cases}
\end{equation}
and
\begin{equation}\label{sigcom}
\begin{aligned}
f(i,A)=\sum_{j>i} a_{h,j}-(-1)^{\delta_{h,m}}\sum_{j>i}a_{h+1,j},\\
g(i,A)=\sum_{j<i}a_{h+1,j}-(-1)^{\delta_{h,m}}\sum_{j<i}a_{h,j}.
\end{aligned}
\end{equation}
\begin{lemma}\label{genmul}
The set $\{ X^{[A]} \,|\, A\in M(m|n)\}$ forms a $\mathbb Q(\up)$-basis for the $U_\up(\mathfrak{gl}_{m|n})$-supermodule $\sSmn$ which has the following actions:
\begin{equation} \notag
\begin{aligned}
(1)\quad K_i.X^{[A]}&=\up_i^{\sum_{1\leq j\leq m}a_{i,j} }X^{[A]},\\
(2)\quad E_h. X^{[A]}&=\sum_{1\leq i\leq m+n\atop\, a_{h+1,i}\geq 1 }
 (-1)^{\sigma_{h,m}(i,A)}\upsh^{f(i,A)}[a_{h,i}+1]X^{[A+E_{h,i}-E_{h+1,i}]},\\
(3)\quad F_h. X^{[A]}&=\sum_{1\leq i\leq m+n\atop a_{h,i}\geq 1}
(-1)^{\sigma_{h,m}(i,A)}\upsi^{g(i,A)}[a_{h+1,i}+1]X^{[A-E_{h,i}+E_{h+1,i}]},\\
\end{aligned}
\end{equation}
where 
\begin{equation}\label{sighm}
\sigma_{h,m}(i,A)=\delta_{h,m}\sigma(i,A)=\begin{cases}\sigma(i,A),&\text{if }h=m;\\
0,&\text{if }h\neq m.\end{cases}
\end{equation}
\end{lemma}
\begin{proof}Let $\Delta^{(N)}=(\Delta\otimes\underbrace{1\otimes\cdots\otimes}_{N-1}1)\circ\cdots\circ(\Delta\otimes1)\circ\Delta$. Then, for  $N=m+n-1$,
\begin{equation}
\begin{aligned}\label{comult}
\Delta^{(N)}(K_i)&= K_i\otimes\cdots\otimes K_i,\\
\Delta^{(N)}(E_h)&= \sum_{i=1}^{m+n} \underbrace{1\otimes 1\cdots \otimes 1}_{i-1} \otimes
E_h\otimes {\widetilde K}_{h}\cdots\otimes
{\widetilde K}_{h},  \\
 \Delta^{(N)}(F_h)&= \sum_{i=1}^{m+n} {\widetilde K}_{h}^{-1}\cdots\otimes
{\widetilde K}_{h}^{-1}  \otimes F_h\otimes\underbrace{1\otimes 1\cdots \otimes
1}_{i-1}.
\end{aligned}
\end{equation}
Thus, by \eqref{eaction} and the sign rule (and, for $h=m$, noting $\up_m=\up,\up_{m+1}=\up^{-1}$),
\begin{equation} \notag
\begin{aligned}
E_h. X^{[A]}&=\Delta^{(N)}(E_h)X^{[A]}\\
&=  \sum_{i=1}^{m+n} (-1)^{\delta_{h,m}(\sum_{j<i} \widehat{\bfc_{j}})}
 (X^{(\bfc_{1})}\otimes\cdots \otimes X^{(\bfc_{i-1})} \otimes E_h. X^{(\bfc_{i})}\\
&\qquad\qquad\qquad\otimes {\widetilde K}_{h}. X^{(\bfc_{i+1})} \cdots \otimes
 {\widetilde K}_{h}.X^{(\bfc_{m+n})})
\\
&=\sum_{1\leq i\leq m+n\atop a_{h+1,i}\geq 1 }
 (-1)^{\sigma_{h,m}(i,A)}\upsh^{f(i,A)}[a_{h,i}+1]X^{[A+E_{h,i}-E_{h+1,i}]}.
\end{aligned}
\end{equation}
The  actions of $F_h, K_i$ can be proved  similarly.
\end{proof}
\begin{remark}We remark that these module formulas are easily obtained, but are the key to the determination of the regular representation of $U_\up(\mathfrak{gl}_{m|n})$. As a comparison, analogous formulas for quantum Schur superalgebras are certain multiplication formuas (see \cite[Props 4.4-5]{DG}) which are obtained by rather lengthy calculations.
\end{remark}
\section{The formal power series algebra $\wsSmn(V)$}
We now extend the module structure on $\sSmn$ to its formal power series algebra and then focus on a submodule which has a $U_\up(\mathfrak{gl}_{m|n})$-supermodule structure. We will   displayed explicitly the actions on a basis.

Recall from \eqref{sSmn} the polynomial superalgebra ${\sSmn}$ and its basis $\{X^{[A]}\}_{\mid A\in M(m|n)}$. By turning the direct sum of all $\mathbb Q(\up)X^{[A]}$ into a direct product, we obtain the {\it formal power series algebra}:
 \begin{equation}\label{wsSmn}
 {\wsSmn}=\wsSmn(V):=\prod_{A\in M(m|n)}\mathbb Q(\up)X^{[A]}\cong\mathbb Q(\up)[[X_{i,j}]]_{1\leq i,j\leq m+n}.
 \end{equation}
For clarity of the $U_\up(\mathfrak{gl}_{m|n})$-actions below, we continue to write the elements in
${\wsSmn}$ by infinite series in $X^{[A]}$'s. Natrually, the $U_\up(\mathfrak{gl}_{m|n})$-action on $\sSmn$ extends to $\wsSmn$ so that $\wsSmn$ becomes a $U_\up(\mathfrak{gl}_{m|n})$-module. We now construct a submodule on which a natural super structure can be built.

Let
$$M(m|n)^\mp=\{A=(a_{i,j})\in M(m|n)\mid  a_{i,i}=0\;\forall i\}.$$
For $\lambda\in\mathbb N^{m+n},A\in M(m|n)^{\mp}$, let $A+\la=A+\diag(\la)$
and, for $\bfj\in\mathbb Z^{m+n}$, define
\begin{equation}\label{A(j)}
 \SgAj=\sum_{\lambda\in \mathbb N^{m+n}}
 \ups^{\lambda* \bfj}X^{[A+\lambda]}
 \in {\wsSmn}.
\end{equation}

Let $\mathcal U(m|n)$ be the subspace of $\wsSmn$ spanned by $A(\bfj)$ for all $A\in M(m|n)^\mp,\bfj\in\mathbb Z^{m+n}$. Since every $X^{[A+\lambda]}$ in $A(\bfj)$ has parity $\widehat A$,
$\mathcal U(m|n)$ has a natural superspace structure $\mathcal U(m|n)=\mathcal U(m|n)_0\oplus\mathcal U(m|n)_1$. In the rest of the section, 
we will prove that $\mathcal U(m|n)$ is a $U_\up(\mathfrak{gl}_{m|n})$-supermodule.

Let $\alpha_h=\bfe_h-\bfe_{h+1}, \beta_h=\bfe_h+\bfe_{h+1}$, $\sigma_{h,m}(i)=\sigma_{h,m}(i,A)$, $f(i)=f(i,A),$ and $ g(i)=g(i,A)$ (see \eqref{sigcom} and \eqref{sighm}).

%
%
\begin{theorem}\label{multi}  The superspace $\mathcal U(m|n)$ is a $U_\up(\mathfrak{gl}_{m|n})$-submodule of $\wsSmn$ with basis
 $\{A(\bfj)\mid A\in M(m|n)^{\mp},\bfj\in \mathbb Z^{m+n}\}$ and the following explicit actions of $E_h,F_h,K_i$:
for $A=(a_{s,t}),\bfj=(j_s)$, $1\leq i\leq m+n$, and $1\leq h< m+n$,
\begin{equation}
\begin{aligned}\label{emulti}
 K_i. \SgAj&= \up_i^{\sum_{1\leq j\leq m+n}a_{i,j}} A(\bfj+\bfe_i),\\
E_h. \SgAj&
= \sum_{i>h+1;\, a_{h+1,i}\geq 1} (-1)^{\sigma_{h,m}(i)}\upsh^{f(i) }[a_{h,i}+1]  (A+E_{h,i}-E_{h+1,i})(\bfj)\\
&+
\sum_{i<h;\, a_{h+1,i}\geq 1}(-1)^{\sigma_{h,m}(i)}
\upsh^{f(i)}[a_{h,i}+1]
      (A+E_{h,i}-E_{h+1,i})(\bfj+\alpha_h)\\
&+\blacktriangleright(-1)^{\sigma_{h,m}(h)}\upsh^{f(h)-j_h
}\frac{(A-E_{h+1,h})(\bfj+\alpha_h)
          -(A-E_{h+1,h})(\bfj-\beta_h)}{\upsh-\upsh^{-1}}\\
&+(-1)^{\sigma_{h,m}(h+1)}\upsh^{f(h+1)
+(-1)^{\delta_{h,m}}j_{h+1}}[a_{h,h+1}+1](A+E_{h,h+1})(\bfj).
\end{aligned}
\end{equation}
\begin{equation}\label{fmulti}
\begin{aligned}
 F_h. \SgAj&
=  \sum_{i<h;\, a_{h,i}\geq 1}  (-1)^{\sigma_{h,m}(i,A)} \upsi^{g(i)}[a_{h+1,i}+1] (A-E_{h,i}+E_{h+1,i})(\bfj)\\
&+\sum_{i>h+1;\, a_{h,i}\geq 1}  (-1)^{\sigma_{h,m}(i,A)} \upsi^{g(i)
}[a_{h+1,i}+1]
     (A-E_{h,i}+E_{h+1,i})(\bfj-\alpha_h)\\
&+ (-1)^{\sigma_{h,m}(h,A)} \upsi^{g(h) +(-1)^{\delta_{h,m}}j_{h}}[a_{h+1,h}+1](A+E_{h+1,h})(\bfj)\\
&+\blacktriangleleft(-1)^{\sigma_{h,m}(h+1,A)} \upsi^{g(h+1)-j_{h+1}
}\frac{(A-E_{h,h+1})(\bfj-\alpha_h)
   -(A-E_{h,h+1})(\bfj-\beta_h)}{\upsi-\upsi^{-1}},
\end{aligned}
\end{equation}
where   $\blacktriangleright$ (resp., $\blacktriangleleft$) is 0 if $a_{h+1,h}=0$ (resp., $a_{h,h+1}=0$), and is 1 otherwise.


Moreover, it is a $U_\up(\mathfrak{gl}_{m|n})$-supermodule.
\end{theorem}
%
%
\begin{proof}The proof of linear independence is similar to that of \cite[Prop. 4.1(2)]{DF2}.

By Lemma \ref{genmul}(1),
\begin{equation}\notag
\begin{aligned}
K_i. \SgAj
&= \sum_{\lambda\in \mathbb N^{m+n}} \ups^{\lambda* \bfj}K_i.X^{[A+\lambda]}= \sum_{\lambda\in \mathbb N^{m+n}}
\ups^{\lambda* \bfj}
      \up_i^{\lambda_i+\sum_{j=1}^{m+n}a_{i,j}}X^{[A+\lambda]}\\
 &= \up_i^{\sum_{j=1}^{m+n}a_{i,j}}
     \sum_{\lambda\in \mathbb N^{m+n}} \ups^{\lambda* ({\bfj+\bfe_i})}
      X^{[A+\lambda]}= \up_i^{\sum_{j=1}^{m+n}a_{i,j}} A(\bfj+\bfe_i).
\end{aligned}
\end{equation}
Similarly, by Lemma \ref{genmul}(2),  and noting $\sigma(i,A)=\sigma(i,A+\la)$,
\begin{equation}\notag
\begin{aligned}
E_h. \SgAj
=& \sum_{\lambda\in \mathbb N^{m+n}} \ups^{\lambda* \bfj}E_h.X^{[A+\lambda]}\\
=&\sum_{\lambda\in \mathbb N^{m+n}}\, \sum_{1\leq i\leq m+n\atop
a_{h+1,i}\geq 1}
    \ups^{\lambda* \bfj }
        (-1)^{\sigma_{h,m}(i)} \upsh^{f(i, A+\lambda)}[a_{h,i}+1]
          X^{[A+\lambda+E_{h,i}-E_{h+1,i}]}\\
=&\sum_{1\leq i\leq m+n\atop
a_{h+1,i}\geq 1}\sum_{\lambda\in \mathbb N^{m+n}}
   (-1)^{\sigma_{h,m}(i)} \ups^{\lambda* \bfj }
        \upsh^{f(i, A+\lambda)}
          [a_{h,i}+1]X^{[A+\lambda+E_{h,i}-E_{h+1,i}]}\\
=&\Sigma_1+\Sigma_2+\Sigma_3+\Sigma_4,
\end{aligned}
\end{equation}
where
$$\aligned
\Sigma_1=&\sum_{i>h+1;\, a_{h+1,i}\geq 1}\sum_{\lambda\in \mathbb N^{m+n}}
         (-1)^{\sigma_{h,m}(i)}\ups^{\lambda* \bfj }  \ups_h^{f(i,A+\lambda) }[a_{h,i}+1]X^{[A+E_{h,i}-E_{h+1,i}+\lambda]},\\
         \Sigma_2=&\sum_{i<h;\, a_{h+1,i}\geq 1}\sum_{\lambda\in \mathbb N^{m+n}}
   (-1)^{\sigma_{h,m}(i)} \ups^{\lambda* \bfj }
        \upsh^{f(i, A+\lambda)}
          [a_{h,i}+1]X^{[A+\lambda+E_{h,i}-E_{h+1,i}]},\\
\Sigma_3=&\sum_{\lambda\in \mathbb N^{m+n}} (-1)^{\sigma_{h,m}(h)}
\ups^{\lambda* \bfj }
         \upsh^{f(h,A+\lambda)}[\lambda_{h}+1]X^{[A-E_{h+1,h}+\lambda+E_{h,h}]},\;\boxed{\text{if }a_{h+1,h}>0},\\
\Sigma_4=&\sum_{\lambda\in \mathbb N^{m+n},\la_{h+1}>0}
(-1)^{\sigma_{h,m}(h+1)}\ups^{\lambda* \bfj }
         \ups_h^{f(h+1,A+\lambda) }[a_{h,h+1}+1]X^{[A+E_{h,h+1}+\lambda-E_{h+1,h+1}]}.
\endaligned$$

By (\ref{sigcom}), for $i\geq h+1$, we have $f(i,A+\la)=f(i,A)$. Thus,
\begin{equation*}
\begin{aligned}
\Sigma_1=&\sum_{i>h+1;\, a_{h+1,i}\geq 1} (-1)^{\sigma_{h,m}(i)}\ups_h^{f(i)
}[a_{h,i}+1]  \sum_{\lambda\in \mathbb N^{m+n}}
 \ups^{\lambda* \bfj } \tiX^{[A+E_{h,i}-E_{h+1,i}+\lambda]}\\
=&\sum_{i>h+1;\, a_{h+1,i}\geq 1} (-1)^{\sigma_{h,m}(i)}\upsh^{f(i)
}[a_{h,i}+1]  (A+E_{h,i}-E_{h+1,i})(\bfj),
\end{aligned}
\end{equation*}
and
\begin{equation*}
\begin{aligned}
\Sigma_4&=(-1)^{\sigma_{h,m}(h+1)}\upsh^{f(h+1) +(-1)^{\delta_{h,m}}\bfe_{h+1}* \bfj}[a_{h,h+1}+1]\cdot\\
&\quad\sum_{\lambda\in \mathbb N^{m+n},\la_{h+1}>0}  \ups^{(\lambda-\bfe_{h+1})* \bfj }
  \tiX^{[A+E_{h,h+1}+\lambda-E_{h+1,h+1}]}\\
&=(-1)^{\sigma_{h,m}(h+1)}\upsh^{f(h+1)
+(-1)^{\delta_{h,m}}j_{h+1}}[a_{h,h+1}+1](A+E_{h,h+1})(\bfj).
\end{aligned}
\end{equation*}
Similarly, for $i<h$, $f(i,A+\la)=f(i,A)+\la_h-(-1)^{\delta_{h,m}}\la_{h+1}$. So, by \eqref{superdp},
\begin{equation*}
\begin{aligned}
\Sigma_2=& \sum_{i<h;\, a_{h+1,i}\geq 1} (-1)^{\sigma_{h,m}(i)}\ups_h^{f(i)
}[a_{h,i}+1] \sum_{\lambda\in \mathbb N^{m+n}}
     \ups^{\lambda* (\bfj +\bfe_h-\bfe_{h+1} )} \tiX^{[A+E_{h,i}-E_{h+1,i}+\lambda]}\\
=& \sum_{i<h;\, a_{h+1,i}\geq 1}(-1)^{\sigma_{h,m}(i)} \upsh^{f(i)}
 [a_{h,i}+1] (A+E_{h,i}-E_{h+1,i})(\bfj+\bfe_h-\bfe_{h+1}).
\end{aligned}
\end{equation*}

Finally, for $\Sigma_3$ when $a_{h+1,h}>0$, since $f(h,A+\la)=f(h,A)-(-1)^{\delta_{h,m}}\lambda_{h+1}$ and
\begin{equation*}
\begin{aligned}
\ups^{\lambda* \bfj }
\upsh^{f(h)-(-1)^{\delta_{h,m}}\lambda_{h+1} }
[\lambda_{h}+1]
&=\upsh^{f(h) }  \ups^{\lambda* (\bfj -\bfe_{h+1})  }
\frac{\upsh^{\lambda_{h}+1}
     -\upsh^{-\lambda_{h}-1}}{\upsh-\upsh^{-1}}\\
&=\upsh^{f(h) }  \frac{\ups^{ \lambda* (\bfj
+\bfe_{h}-\bfe_{h+1})}\upsh
     -\ups^{\lambda* ( \bfj -\bfe_{h}-\bfe_{h+1}) } \upsh^{-1}}{\upsh-\upsh^{-1}}\\
&=\upsh^{f(h) }  \frac{\ups^{ \lambda*
(\bfj+\bfe_h-\bfe_{h+1})+\bfe_h*\bfe_h}
     -\ups^{\lambda* (\bfj-\bfe_h-\bfe_{h+1})-\bfe_h*\bfe_h}}{\upsh-\upsh^{-1}}\\
&=\upsh^{f(h) }  \frac{\ups^{ (\lambda+\bfe_h)*
(\bfj+\bfe_h-\bfe_{h+1})}\upsh^{-j_h}
     -\ups^{(\lambda+\bfe_h)* (\bfj-\bfe_h-\bfe_{h+1})
     }\upsh^{-j_h}}{\upsh-\upsh^{-1}}\\
&=\upsh^{f(h)-j_h } \, \frac{\ups^{ (\lambda+\bfe_h)*
(\bfj+\bfe_h-\bfe_{h+1})}
                  -\ups^{(\lambda+\bfe_h)* (\bfj-\bfe_h-\bfe_{h+1})}}{\upsh-\upsh^{-1}},
\end{aligned}
\end{equation*}
it follows that
\begin{equation}\notag
\begin{aligned}
\Sigma_3=&\sum_{\lambda\in \mathbb N^{m+n}} (-1)^{\sigma_{h,m}(h)}
\ups^{\lambda* \bfj }
         \upsh^{f(h)-(-1)^{\delta_{h,m}}\lambda_{h+1} }[\lambda_{h}+1]\tiX^{[A-E_{h+1,h}+\lambda+E_{h,h}]}\\
\!\!\!=&(-1)^{\sigma_{h,m}(h)}\upsh^{f(h)-j_h }\!\!\sum_{\lambda\in \mathbb
N^{m+n}}
      \frac{\ups^{(\lambda+\bfe_h)\cdot (\bfj+\bfe_h-\bfe_{h+1})}
        -\ups^{(\lambda+\bfe_h)\cdot (\bfj-\bfe_h-\bfe_{h+1})
}}{\upsh-\upsh^{-1}} \tiX^{[A-E_{h+1,h}+\lambda+\bfe_{h}]}\\
\!\!\!=&(-1)^{\sigma_{h,m}(h)}\upsh^{f(h)-j_h
}\frac{(A-E_{h+1,h})(\bfj+\bfe_h-\bfe_{h+1})
          -(A-E_{h+1,h})(\bfj-\bfe_h-\bfe_{h+1})}{\upsh-\upsh^{-1}},
\end{aligned}
\end{equation}
proving \eqref{emulti}. (Notice a cancellation for the terms associated to those $\la$ with $\la_h=0$ when expanding the numerator of the last expression.)

 The proof for the action of $F_h$ is similar. Finally, the supermodule assertion follows easily from the action formulas.
\end{proof}

\section{The main result}
We are now ready to prove the main result of the paper by the following.

\begin{lemma}\label{key lemma} Let $U$ be an algebra over a field $\bbmk$ with generators $g_i$, $1\leq i\leq N$. Suppose $Uv$ is a cyclic $U$-module with basis $b_j=u_j.v,\;j\in J$ ($u_j\in U$), and trivial annihilator $\text{\rm ann}_U(v)=0$. Then the matrix representations
$g_i.b_j=\sum_{k\in J}\la_{i,j,k}b_k$
 of the generators give rise to a presentation of $U$ by basis $\{u_j\mid j\in J\}$ and the multiplication formulas:
$$g_iu_j=\sum_{k\in J}\la_{i,j,k}u_k,\quad\text{ for all $1\leq i\leq N,j\in J$}.$$
\end{lemma}
\begin{proof}Since the $U$-module homomorphism $\phi:U\to Uv, u\mapsto u.v$ is an isomorphism, the basis claim is clear and so are the multiplication formulas.
 \end{proof}

For
$A=(a_{i,j})\in M(m|n)$, let
$$A^{\textsf L}_{s,t}=\sum_{i\leq
s,j\geq t}a_{i,j}\text{ if $s<t$}, \quad A^\neg_{s,t}= \sum_{i\geq
s,j\leq t}a_{i,j}\text{ if $s>t$}.$$
Following \cite[\S3.5]{BLM} or \cite[(8.0.1)]{DG}, define a preorder relation on $M(m|n)$:
 \begin{equation}\label{prec}\notag
 A\preceq B\iff
\begin{cases}
A_{s,t}^{\textsf L}\leq B_{s,t}^{\textsf L},&\text{for all $s<t$};\\
A_{s,t}^\neg\leq B_{s,t}^\neg,&\text{for all $s>t$}.\end{cases}
\end{equation}
Note that this is a partial order relation on $M(m|n)^\mp$. The $U_\up(\mathfrak{gl}_{m|n})$-actions in Theorem \ref{multi} satisfy certain ``triangular relations'' relative to $\preceq$. The ``lower terms'' below means a linear combination of $B(\bfj')$ with $B\prec$ the leading matrix.

\begin{lemma}\label{tri relation}Let $A=(a_{i,j})\in M(m|n)^\mp$, $\bfj\in\mathbb Z^{m+n}$, and $h,k\in[1,m+n]$.
\begin{enumerate}
\item If $h<k$, {\rm $A^{\textsf {L}}_{h+1,k+1}=0$},  $a_{h,k}=0$, and $a_{h+1,k}\geq a>0$, then, for some $b\in\mathbb Z$,
$$E_h^{(a)}.A(\mathbf j)=\pm \up^b (A+aE_{h,k}-aE_{h+1,k})(\mathbf j)+(\text{lower terms}).$$
\item If $h+1>k$, $A^\neg_{h,k-1}=0$, $a_{h+1,k}=0$, and $a_{h,k}\geq a>0$, then, for some $c\in\mathbb Z$,
$$F_h^{(a)}.A(\mathbf j)=\pm \up^c (A-aE_{h,k}+aE_{h+1,k})(\mathbf j)+(\text{lower terms}).$$
\end{enumerate}
\end{lemma}
\begin{proof}This follows easily from repeatedly applying the actions in Theorem \ref{multi}. For example, the first summation in $E_h.A(\mathbf 0)$ contains only the terms $(A-E_{h,i}+E_{h+1,i})(\mathbf 0)$, for some $h+1<i\leq k$, and $A+E_{h,k}-E_{h+1,k}\succ A+E_{h,i}-E_{h+1,i}$ for all $i<h$ or $h+1<i<k$ if it occurs in the first two summations. One sees also $A+E_{h,k}-E_{h+1,k}\succ A-E_{h+1,h}, A+E_{h,h+1}$. Hence, $E_h.A(\mathbf 0)=\pm \up^* (A+E_{h,k}-E_{h+1,k})(\mathbf 0)+(\text{lower terms}).$ Inductively, $E_h^a.A(\mathbf 0)=\pm \up^b[a]^! (A+aE_{h,k}-aE_{h+1,k})(\mathbf 0)+(\text{lower terms}).$ Hence, the desired formulas follows.
\end{proof}

\begin{theorem}\label{genmod} The $U_\up(\mathfrak{gl}_{m|n})$-supermodule $\mathcal U(m|n)$ is a cyclic module generated by  $O(\bf 0)$, where $O\in M(m|n)$ and $\bf0\in\mathbb N^{m+n}$ are the zero elements,
 and the module homomorphism
\begin{equation}\label{m h}
f:U_\up(\mathfrak{gl}_{m|n})\longrightarrow \mathcal U(m|n),u\longmapsto u.O(\bf0).
\end{equation}is an isomorphism.
\end{theorem}
\begin{proof}By Lemma \ref{tri relation}, we may use an argument similar to that for \cite[Proposition 3.9]{BLM}). Consider reduced expressions of the longest elements in the symmetric groups $\fS_{\{1,2,\ldots,j\}}$ for $j=2,3,\ldots,m+n$ and $\fS_{\{k,k+1,\ldots,m+n\}}$ for $k=1,2,\ldots,m+n-1$:
$$s_{j-1}(s_{j-2}s_{j-1})\cdots(s_1s_2\cdots s_{j-1}),\quad s_{k}(s_{k+1}s_{k})\cdots(s_{m+n-1}\cdots s_{k+1}s_{k}).$$
For any $A=(a_{i,j})\in M(m|n)^\mp$ and $\bfj\in\mathbb Z^{m+n}$, let
    $$\aligned
    \mathfrak m^+_j&=\mathfrak m^+_j(A)=E_{j-1}^{(a_{j-1,j})}(E_{j-2}^{(a_{j-2,j})}E_{j-1}^{(a_{j-2,j})})
\cdots (E_1^{(a_{1,j})}E_2^{(a_{1,j})}\cdots E_{j-1}^{(a_{1,j})}),\\
 \mathfrak m^-_k&=\mathfrak m^-_k(A)=F_{k}^{(a_{k+1,k})}(F_{k+1}^{(a_{k+2,k})}F_{k}^{(a_{k+2,k})})\cdots(F_{m+n-1}^{(a_{m+n,k})}\cdots F_{k+1}^{(a_{m+n,k})}F_k^{(a_{m+n,k})}),\\
\endaligned$$
and let $ \mathfrak m^{A,\bfj}=\mathfrak m^-_\bullet K^\bfj \mathfrak m^+_\bullet,$ where
$\mathfrak m^-_\bullet=\mathfrak m^-_1\mathfrak m^-_2\cdots \mathfrak m^-_{m+n-1}$, $K^\bfj=K_1^{j_1}\cdots K_{m+n}^{j_{m+n}} $, and $\mathfrak m^+_\bullet=\mathfrak m^+_{m+n}\cdots \mathfrak m^+_3 \mathfrak m^+_2$.
For example, if $m=2,n=2,$ $A\in M(2|2)^{\pm}, \bfj\in \Z^{4}$, then
\begin{equation*}\label{tdef1}
\begin{aligned}
\mathfrak m^{A,\bfj}&=(F_1^{(a_{21})}
F_2^{(a_{31})} F_1^{(a_{31})}
F_3^{(a_{41})}F_2^{(a_{41})} F_1^{(a_{41})})
(F_2^{(a_{32})}
F_3^{(a_{42})} F_2^{(a_{42})})
F_3^{(a_{43})}
K_1^{\bfj_1}K_2^{\bfj_2}K_3^{\bfj_3}K_4^{\bfj_4}\\
&\quad\,(E_3^{(a_{34})}
E_2^{(a_{24})} E_3^{(a_{24})}
E_1^{(a_{14})} E_2^{(a_{14})} E_3^{(a_{14})})
(E_2^{(a_{23})}
E_1^{(a_{13})} E_2^{(a_{13})})
E_1^{(a_{12})}.
\end{aligned}
\end{equation*}
Repeatedly applying Lemma \ref{tri relation}, we obtain
$$
\mathfrak m^{A,\bfj}.O({\bf 0})=\pm \up^cA({\bf j})+\mbox{(lower terms)}\quad (c\in\mathbb Z).
$$
In fact, $\mathfrak m^+_2.O(\bf0)$ has the leading term $(a_{1,2}E_{1,2})(\bf0)$, $\mathfrak m^+_3\mathfrak m^+_2.O(\bf0)$ has the leading term $(a_{1,2}E_{1,2}+a_{1,3}E_{1,3}+a_{2,3}E_{2,3})(\bf 0)$, ... , $\mathfrak m^+_{m+n}\cdots \mathfrak m^+_3 \mathfrak m^+_2.O(\bf0)$ has the leading term $A^+(\bf 0)$, where $A^+$ is the upper triangular part of $A$. Similarly, $K^\bfj\mathfrak m^+_\bullet.O(\bf0)$ has the leading term $A^+(\bfj)$, $\mathfrak m^-_{m+n-1}K^\bfj\mathfrak m^+_\bullet.O(\bf0)$ has the leading term $(A^++a_{m+n,m+n-1}E_{m+n,m+n-1})(\bfj)$, and so on.

Since $\{A(\bfj)\mid A\in M(m|n)^{\mp},\bfj\in \mathbb N^{m+n}\}$ forms a basis for $\mathcal U(m|n)$ by Theorem \ref{multi}, the triangular relation above implies that $\{\mathfrak m^{A,\bfj}.O({\bf 0})\mid A\in M(m|n)^{\mp},\bfj\in \mathbb N^{m+n}\}$ are linearly independent. Hence, the module homomorphism
\eqref{m h}
must be an isomorphism.
\end{proof}

The theorem above gives immediately a presentation for $U_\up(\mathfrak{gl}_{m|n})$.

\begin{corollary}\label{mthm} The supergroup $U_\up(\mathfrak{gl}_{m|n})$ contains a basis
$$\{A(\bfj)\mid A\in M(m|n)^{\mp},\bfj\in \mathbb Z^{m+n}\}$$  such that $E_h=E_{h,h+1}(\bf 0)$
$F_h=E_{h+1,h}(\bf 0)$, and $K_i=O(\bfe_i)$, and the $U_\up(\mathfrak{gl}_{m|n})$-action formulas given in Theorem \ref{multi} become the multiplication formulas of the basis elements $A(\bfj)$ by the generators.
\end{corollary}

\begin{proof} By the module isomorphism \eqref{m h} (and by abuse of notation), let $A(\bfj):=f^{-1}A(\bfj)$. Since $E_h.O(\mathbf0)=E_{h,h+1}(\bf0)$, $F_h.O(\mathbf0)=E_{h+1,h}(\bf0)$, and $K_i.O(\mathbf 0)=O(\bse_i)$, we have $E_h=E_{h,h+1}(\bf0)$, $F_h=E_{h+1,h}(\bf0)$, and $K_i=O(\bse_i)$. The assertion now follows from Lemma \ref{key lemma}.
\end{proof}

\begin{remark}The presentation above for $U_\up(\mathfrak{gl}_{m|n})$ coincides with the one from \cite[Lemma 5.3]{BLM} (or \cite[Theorem 14.8]{DDPW}) in the quantum $\mathfrak {gl}_{n}$ case and with the one in \cite[Thm 8.4]{DG} in general after a sign modification given below.
\end{remark}

For any $A=(a_{i,j})\in M(m|n)$, let\footnote{This number $\overline{A}$ is different from the number $\bar A$ defined in \cite[(5.0.1)]{DG}, where the super grading structure on the tensor space is under consideration.} 
\begin{equation}\label{signAbar}
\overline{A}=\sum_{\substack{1\leq i,\, k\leq m
\\ m<j<l\leq m+n}}a_{i,j}a_{k,l}.
\end{equation}

\begin{lemma}\label{Abar}
For $\lambda\in\mathbb N^{m+n},A=(a_{i,j})\in M(m|n)^{\mp}$ and
 $1\leq h,k\leq m+n$ with $h<m+n$, then 
 \begin{itemize}
 \item[(1)]
 $\overline{A+\lambda}
=\overline{A}$;
\item[(2)]$\overline{A}+\delta_{h,m}\sigma(k,A)
= \overline{A+E_{h,k}-E_{h+1,k}}
+{\delta_{h, m}\displaystyle\bigg(\sum_{{\substack{ i> m \\ j\leq min\{k-1, m\}}}}  a_{i,j}
-{\delta^>_{k,m}\sum_{{\substack{ i\leq m \\ j>k}}}  a_{i,j}\bigg)}}.$
\end{itemize}
Here, $\delta^>_{k,m}=1$  if $k>m$ and $0$ otherwise.
\end{lemma}
\begin{proof}If we write $A={X\;\;Q\choose Q'\;\; Y}$ in blocks as in \eqref{blocks}, then the entries involved in $\overline {A}$ are all in $Q$. Thus, (1) and (2) for $h\neq m$ or $h=m,k\leq m$ are all clear. Assume now $h=m, k>m$. Then, by definition,
\begin{equation*}
\begin{aligned}
\overline{A+E_{m,k}-E_{m+1,k}}
&=\overline{A}+\sum_{ i\leq m, m<j<k}  a_{i,j}
+\sum_{ i\leq m,j>k}  a_{i,j}\\
&=\overline{A}+(\sum_{i\leq m, m<j<k}  a_{i,j}
+\sum_{i> m,j\leq m}  a_{i,j})
-(\sum_{i> m,j\leq m}  a_{i,j}
-\sum_{i\leq m, j>k}  a_{i,j}
)\\
&=\overline{A}+\sigma(k,A)
-(\sum_{i> m, j\leq m}  a_{i,j}
-\sum_{i\leq m, j>k}  a_{i,j}
),
\end{aligned}
\end{equation*}
as desired.
\end{proof}

Let $\ulA(\bfj) =(-1)^{\bar A}  A(\bfj)$ for all $ A\in M(m|n)^{\mp},\bfj\in \mathbb Z^{m+n}$.

\begin{theorem}\label{compar}
Modifying the multiplication formulas in Theorem \ref{multi} by using the basis $\{\ulA(\bfj)\mid A\in M(m|n)^{\mp},\bfj\in \mathbb Z^{m+n}\}$ for the supergroup $U_\up(\mathfrak{gl}_{m|n})$ yields exactly the same formulas as given in \cite[Thm 8.4]{DG}.
\end{theorem}

\begin{proof}We first observe that the generators $E_h=E_{h,h+1}({\bf0})=\underline{E_{h,h+1}}(\bf0)$, etc. are part of the new basis. After multiplying both sides of the multiplication formulas in  Theorem \ref{multi} by $(-1)^{\overline {A}}$ and applying Lemma \ref{Abar}, the sign term becomes
$(-1)^{s(h,i)}$ with 
$$s(h,i)={\delta_{h, m}(\sum_{ s> m, t\leq min\{i-1, m\}}  a_{s,t}
+{\delta^>_{i,m}\sum_{ s\leq m, t>i}  a_{s,t})}}$$
This number $s(h,i)$ is exactly the same number $\varepsilon_{h,h+1}\sigma(i)$ defined in \cite[(5.5.1-2)]{DG} and used in the multiplication formulas in \cite[Thm 8.4]{DG}.
\end{proof}


\noindent
{\bf Acknowledgement.} The authors would like to thank the referee for a correction on the parity computation involved in Lemma \ref{genmul}. This eventually led to a significant improvement of the paper.

\end{document}